\documentclass[openany, amssymb, psamsfonts]{amsart}
\usepackage{mathrsfs,comment}
\usepackage[usenames,dvipsnames]{color}
\usepackage[normalem]{ulem}
\usepackage{url}
\usepackage[all,arc,2cell]{xy}
\usepackage{graphicx}
\UseAllTwocells
\usepackage{enumerate}
\usepackage{hyperref}  
\hypersetup{%
  bookmarksnumbered=true,%
  bookmarks=true,%
  colorlinks=true,%
  linkcolor=blue,%
  citecolor=blue,%
  filecolor=blue,%
  menucolor=blue,%
  pagecolor=blue,%
  urlcolor=blue,%
  pdfnewwindow=true,%
  pdfstartview=FitBH}

\def\makeautorefname#1#2{\expandafter\def\csname#1autorefname\endcsname{#2}}
\def\equationautorefname~#1\null{(#1)\null}
\makeautorefname{footnote}{footnote}%
\makeautorefname{item}{item}%
\makeautorefname{figure}{Figure}%
\makeautorefname{table}{Table}%
\makeautorefname{part}{Part}%
\makeautorefname{appendix}{Appendix}%
\makeautorefname{chapter}{Chapter}%
\makeautorefname{section}{Section}%
\makeautorefname{subsection}{Section}%
\makeautorefname{subsubsection}{Section}%
\makeautorefname{theorem}{Theorem}%
\makeautorefname{thm}{Theorem}%
\makeautorefname{cor}{Corollary}%
\makeautorefname{lem}{Lemma}%
\makeautorefname{prop}{Proposition}%
\makeautorefname{pro}{Property}
\makeautorefname{conj}{Conjecture}%
\makeautorefname{defn}{Definition}%
\makeautorefname{notn}{Notation}
\makeautorefname{notns}{Notations}
\makeautorefname{rem}{Remark}%
\makeautorefname{quest}{Question}%
\makeautorefname{exmp}{Example}%
\makeautorefname{ax}{Axiom}%
\makeautorefname{claim}{Claim}%
\makeautorefname{ass}{Assumption}%
\makeautorefname{asss}{Assumptions}%
\makeautorefname{con}{Construction}%
\makeautorefname{prob}{Problem}%
\makeautorefname{warn}{Warning}%
\makeautorefname{obs}{Observation}%
\makeautorefname{conv}{Convention}%

\newtheorem{thm}{Theorem}[section]

\newtheorem{prop}{Proposition}[section]
\newtheorem{lem}{Lemma}[section]

\theoremstyle{definition}
\newtheorem{defn}{Definition}[section]

\newtheorem{exmp}{Example}[section]
\newtheorem{notn}{Notation}[section]

\makeatletter
\let\c@obs=\c@thm
\let\c@cor=\c@thm
\let\c@prop=\c@thm
\let\c@lem=\c@thm
\let\c@prob=\c@thm
\let\c@con=\c@thm
\let\c@conj=\c@thm
\let\c@defn=\c@thm
\let\c@notn=\c@thm
\let\c@notns=\c@thm
\let\c@exmp=\c@thm
\let\c@ax=\c@thm
\let\c@pro=\c@thm
\let\c@ass=\c@thm
\let\c@warn=\c@thm
\let\c@rem=\c@thm
\let\c@sch=\c@thm
\let\c@equation\c@thm
\numberwithin{equation}{section}
\makeatother

\title{The Landscape of Large Cardinals}

\author{Rohan Srivastava}

\date{Submitted August 31, 2021 as part of the Freiwald Scholars Summer REU Program at Washington University in St. Louis. Revised: September 27, 2021, February 12, 2022, April 25, 2022.}

\begin{document}

\begin{abstract}

The purpose of this paper is to provide an introductory overview of the large cardinal hierarchy in set theory. By a large cardinal, we mean any cardinal $\kappa$ whose existence is strong enough of an assumption to prove the consistency of ZFC. We assume basic familiarity with set theory, model theory, and the ZFC axioms, though certain concepts will be reviewed as necessary. We attempt to clarify the vast landscape of the higher infinite, motivating the definitions of some (but certainly not all) of the known large cardinals. We also discuss connections with the Continuum Hypothesis and provide some philosophical reflections on belief in the consistency of large cardinals.

\end{abstract}

\maketitle

\tableofcontents

\textit{``How many nights does it take to count the stars? If I tried, I know it would feel like $\omega$." -Harry Styles}

\section{Introduction}  
\subsection{ZFC preliminaries}
Set theory serves as an ontological and epistemological foundation for all of mathematics, but it is also a branch of mathematics in its own right. It begins with, and is the study of, a particular mathematical object: the von Neumann universe, stratified by the \textit{cumulative hierarchy} of sets. We denote the set-theoretic universe by V, and define it by transfinite recursion as follows:

$V_0 = \emptyset$

$V_{\alpha + 1} = \mathcal{P}(V_{\alpha})$

$V_{\beta} = \cup\{V_{\alpha} : {\alpha < \beta}\},$ for limit ordinals $\beta.$

$V=\cup\{V_\alpha : \alpha\in ON\},$ where ON denotes the class of all ordinals

For thoroughness, we also briefly review the ZFC axioms of set theory, stating a few formally and the others only informally.
\newline

\noindent Axiom 1. Extensionality.
\begin{itemize}
    \item $\forall z(z\in x \leftrightarrow z\in y) \rightarrow x=y$. \footnote{That is, sets are defined by their elements. Note that in our notation all free variables (i.e., those not bounded by quantifiers, such as x and y here) are implicitly taken to be universally quantified.}
\end{itemize}
 Axiom 2. Foundation.
\begin{itemize}
    \item $\exists y(y\in x) \rightarrow \exists y(y\in x \wedge \neg\exists z(z\in x \wedge z\in y))$
\end{itemize}
 Axiom 3. Comprehension Scheme.\footnote{Note that this is not a single axiom but an infinite schema of axioms, one corresponding to each formula. As something of a curiosity, the Comprehension scheme actually can be replaced by finitely many instances of Comprehension to obtain an equivalent theory. However, this does not hold true of the other infinite axiom schema in ZFC, Replacement, and thus ZFC is not a finitely axiomatizable theory.}
\begin{itemize}
    \item For each formula $\varphi$ without y free,
    \newline
    $\exists y\forall x(x\in y \leftrightarrow x\in v \wedge \varphi(x))$
\end{itemize}
 Axiom 4. Pairing. 
\begin{itemize}
    \item For any two sets A and B, there exists a set C containing exactly A and B as members.
\end{itemize}
 Axiom 5. Union. 
\begin{itemize}
    \item For any given set A, there exists a set containing exactly the elements of elements of A.
\end{itemize}
 Axiom 6. Replacement Scheme.
\begin{itemize}
    \item The image of any set under any (definable) mapping is a set.
\end{itemize}
 Axiom 7. Infinity. 
\begin{itemize}
    \item There exists some infinite set, namely $\omega$.
\end{itemize}
 Axiom 8. Power Set. 
\begin{itemize}
    \item Given a set A, there exists a set whose elements are all the subsets of A.\footnote{Or, perhaps more naturally stated, the ``set of all subsets" is indeed a set.}
\end{itemize}
 Axiom 9. Choice. 
\begin{itemize}
    \item Given any collection X of non-empty sets, there exists a \textit{choice function} on X mapping each set of X to one of its elements.
\end{itemize}

Historically speaking, Zermelo backed out an early formulation of the ZFC axioms in 1908 from the implicit hypotheses assumed in his proof of the Well-Ordering Theorem. This predates the conception of their canonical model, V, which was described in 1930. The relationship between the ZFC axioms and the the von Neumann universe is murky however, and philosophies differ on this. Platonists will contend that the von Neumann universe is something concrete; certain things (e.g. the Continuum Hypothesis) hold true of it and others are false of it. On this view, the truth-value of the Continuum Hypothesis can be resolved by adding new axioms. On the other hand, those with formalist tendencies may contend that the description of the von Neumann universe is inherently vague; we know that $V\models ZFC$, but we do not, for instance, really know what the continuum $2^{\aleph_0}$ is because there are many things in it that we can never even describe in a finite language.

\subsection{Reconciling completeness with incompleteness (the Dez Bryant section)}

In set theory, we associate ``existence" with ``being a set." Note that in our ontology then V is a \textit{proper class} (i.e., not a set), but informally we may think of it as the structure consisting of all sets.\footnote{It is easily seen as a consequence of Russell's paradox (or even of Cantor's theorem) that V cannot itself be a set.} Godel's Completeness Theorem states that if a sentence is true in all models of a first order theory, then it can be proven; an equivalent formulation is that if a theory is consistent, then it has a model. But this model must provably ``exist" (i.e., be a set), so this does not contradict the fact that $V\models ZFC$. However, by restricting our attention to certain stages $V_\lambda$ of the cumulative hierarchy, we can obtain certain nice models which are sets. For instance, $V_{\omega}$, the set of all hereditarily finite sets (sets whose elements are finite, and elements of elements are finite, and so on), is in fact a model of all the ZFC axioms except for the Axiom of Infinity. Similarly, $V_{\omega+\omega}$ is a model of all of the ZFC axioms except for Replacement. Unfortunately, by Gödel's Second Incompleteness Theorem,\footnote{Gödel's Second Incompleteness Theorem says, roughly, that any first-order theory with a decidable (i.e., recognizable and co-recognizable by a Turing machine) set of axioms sufficient to prove arithmetical statements cannot prove its own consistency.} we can never prove that any initial segment $V_\kappa$ of the set-theoretic universe can model all of ZFC (or else ZFC, by proving $V_\kappa$ exists as a set, would prove its own consistency). Thus large cardinal axioms, postulating the existence of large cardinals $\kappa$ (typically with the property that $V_\kappa\models ZFC$\footnote{By the definition we take here of a large cardinal, a large cardinal $\kappa$ need only have the property that its existence is strong enough of an assumption (along with the rest of ZFC) to prove the consistency of ZFC. In practice, all such large cardinals have the property that $V_\kappa\models ZFC$ so some texts choose instead to treat this as the definition of a large cardinal.}), hold a special epistemic status: they are (provably) not provable, but they are falsifiable.\footnote{By analogy, large cardinal axioms may be thought of as ``strong axioms of infinity."} In other words, there is no hope of finding a proof that any given large cardinal is consistent, but it may be that we find a proof that the existence of some kind of large cardinal leads to a contradiction. Nevertheless, few working set theorists worry too much about this, for reasons that we will discuss later.

We shall now begin to define some large cardinal notions, beginning with the most elementary.

\begin{defn}  A cardinal $\kappa$ is \textbf{weakly inaccessible} if $\kappa > \omega$, $\kappa$ is regular, and $\kappa > \lambda^+$ for all $\lambda < \kappa$. 
\end{defn}

In other words, weakly inaccessible is equivalent to being a regular limit cardinal, and implies being a fixed point of the Aleph function; that is, $\kappa = \aleph_\kappa$. (We recall that a \textit{regular cardinal} is one who cofinality is equal to itself and a \textit{limit cardinal} is a cardinal $\aleph_\lambda$ such that $\lambda$ is a limit ordinal.)

\begin{defn} A cardinal $\kappa$ is \textbf{strongly inaccessible} if $\kappa$ is weakly inaccessible and also satisfies the property: if $\kappa > \lambda$ then $\kappa > 2^\lambda$.

\end{defn}

Thus strongly inaccessible cardinals are unreachable even by  the power set operation. Note that assuming the Generalized Continuum Hypothesis (that $2^{\lambda} = \lambda^+$ for all cardinals $\lambda$), the notions of strong and weak inaccessibility coincide. Regardless, strong inaccessibility is the preferred notion, so by an ``inaccessible" cardinal we shall usually mean a strongly inaccessible cardinal. It is also worth noting that Cohen's method of \textit{forcing}, first used to prove the consistency of the negation of the Continuum Hypothesis with the ZFC axioms,\footnote{In conjunction with Gödel's construction of L and proof that L satisfies CH, this established the independence of CH from the ZFC axioms. In fact, forcing can be used to establish the consistency of CH as well as $\neg$CH.} is so robust that it can be used even to establish that the continuum $2^{\aleph_0}$ may be a weakly inaccessible cardinal.

Inaccessible cardinals can be thought of as sitting ``on top" of the cumulative hierarchy, and in particular they are large cardinals: $V_\kappa \models ZFC$ for any inaccessible cardinal $\kappa$. But if you believe in the consistency of ZFC (as we all seem to, given that it serves as a foundation for all of mathematics), it is really not too much stronger to postulate the existence of some inaccessible cardinal.\footnote{Technically, though, the existence of an inaccessible cardinal is a strictly stronger assumption than the existence of a transitive model of ZFC, which is in turn a strictly stronger assumption than the consistency of ZFC (i.e. the existence of any model of ZFC)} The same cannot be said for some of the stronger large cardinal hypotheses which we shall see.

\section{Small large cardinals and the constructible universe}

We can generalize the notion of inaccessibility in the following manner:

\begin{defn}
\item $\kappa$ is \textit{0-weakly inaccessible} iff $\kappa$ is a regular limit cardinal.

\item $\kappa$ is \textit{($\alpha$ + 1)-weakly inaccessible} iff $\kappa$ is a regular limit of $\alpha$-weakly inaccessible cardinals.

\item For limit ordinals $\delta$, $\kappa$ is \textit{$\delta$-weakly inaccessible} iff $\kappa$ is $\alpha$-weakly inaccessible for every $\alpha<\delta$.
\end{defn}

The \textbf{Mahlo hierarchy} is defined analogously, but first we need two preliminary definitions:

\begin{defn}
A subset S of a limit ordinal $\delta$ is a \textbf{stationary} set (in $\delta$) if it has non-empty intersection with every \textbf{club} (closed unbounded) set in $\delta$.
\end{defn}

\begin{defn}

\item $\kappa$ is \textit{0-weakly Mahlo} iff $\kappa$ is regular.
    
\item $\kappa$ is \textit{$(\alpha + 1)$-weakly Mahlo} iff $\{\xi < \kappa : \xi$ is $\alpha$-weakly Mahlo$\}$ is stationary in $\kappa$.

\item For limit ordinals $\delta$, $\kappa$ is \textit{$\delta$-weakly Mahlo} iff $\kappa$ is $\alpha$-weakly Mahlo for every $\alpha < \delta$.

\item $\kappa$ is \textbf{strongly Mahlo} (or simply Mahlo) iff $\{\alpha < \kappa : \alpha $ is inaccessible$\}$ is stationary in $\kappa$.

\end{defn}

Note that being a stationary set is a notion of ``largeness," so being Mahlo entails being ``so large" that the number of inaccessibles below is itself ``large." This is remarkable given that postulating the existence of a single additional inaccessible always establishes a stronger theory.\footnote{For instance, ZFC + ``there exists 87 inaccessible cardinals" proves the consistency of ZFC + ``there exist exactly 86 inaccessible cardinals." The model is $V_{\text{the 87th}}$.} Nevertheless, in the context of the large cardinal hierarchy as a whole, Mahlo cardinals are still quite weak in terms of relative consistency strength.

In the literature, there is a distinction between ``small" large cardinals and ``large" large cardinals.\footnote{This is unfortunate terminology for reasons that will be discussed in section 6.} This divide has everything to do with Scott's Theorem, which we shall arrive at shortly. First, we must review Godel's constructible universe, denoted L.

The L-hierarchy is defined similarly to the cumulative hierarchy, replacing the full powerset with the ``definable" powerset at the successor stage in the tranfinite recursion.\footnote{The definable powerset of a set A, denoted $\mathcal{D^+}(A)$ contains only those subsets of A that can be specified by a formula using some finite number of parameters taken from A.}

\begin{itemize}
    \item $L(0)=\emptyset$
    \item $L(\beta + 1) = \mathcal{D^+}(L(\beta))$
    \item $L(\gamma) = \cup_{\alpha < \gamma} L(\alpha)$ for limit ordinals $\gamma$
    
\noindent And finally,

    \item $L = \cup_{\alpha\in ON} L(\alpha)$.
\end{itemize}

Gödel established the consistency of the Continuum Hypothesis with ZFC by showing (quite nontrivially) that L is a (proper class) model of all of the ZFC axioms as well as CH. In fact, even more strongly, it can be shown in ZF that L is a model of GCH, and we know that L is the minimal class containing all the ordinals that models ZF.\footnote{ZF is the standard notation for ZFC without the axiom of choice. Furthermore, Sierpiński showed in ZF that GCH implies AC, so by adding the hypothesis V=L to the ambient theory ZF, we freely obtain GCH and AC as consequences.} It is also consistent for there to exist $\delta$-weakly inaccessible and even Mahlo cardinals in L (assuming that it is consistent for them to exist in V). One may think, then, that a natural candidate axiom to augment ZFC by would by the Axiom of Constructibility, stating that V=L. After all, this is consistent with the axioms of set theory including certain large cardinals, and, by implying GCH and AC, cleans up the picture of the cumulative hierarchy considerably. Indeed, in asserting that there cannot be two different standard models of V containing all the same ordinals, the Axiom of Constructibility restricts the universe to a mathematical object that we can much more reasonably call ``concrete." Unfortunately, as Scott showed and as we shall see, the situation is not so simple...

\section{Ultrafilters and ultraproducts}

We shall eventually give two conceptual characterizations of one of the most widely studied of all large cardinals: a measurable cardinal. To give the first-order one that is of a more combinatorial flavor, we first introduce some machinery.

\begin{defn}
   Let I be a nonempty set. A filter D over I is defined to be a set $D\subseteq \mathcal{P}(I)$ such that:
        \begin{itemize}
            \item (Full set) $I\in D$
            \item (Closure under finite intersections) If $X, Y\in D$, then $X\cap Y\in D$
            \item (Upwards closure) If $X\in D$ and $X\subset Z\subset I$, then $Z\in D$.
        \end{itemize}
\end{defn}

A filter D is a \textit{proper filter} if it is not $\mathcal{P}(I)$. Additionally, a filter is said to be \textit{principal} if $\exists Y\subset I$ such that $D = \{X\subset I : Y\subset X\}$. As expected, a filter is \textit{nonprincipal} if it is not principal. A canonical example of a nonprincipal filter is the set of all cofinite subsets of an infinite set I—this is called the \textit{Fréchet filter}.

\begin{prop}
A filter D is principal iff $\cap D\in D$.
\end{prop}
\begin{proof}
Let D be a filter over I. First, suppose that $A=\cap D\in D$. Let $F_A$ be the principal filter generated by A. Then clearly $D\subset F_A$. Let $Y\in F_A$ and suppose $Y\notin D$. But then since $A\subset Y$, $A\notin D$ as well, which is a contradiction. So $F_A\subset D$ and hence $D=F_A$; thus D is principal. Conversely, suppose D is principal. Then $D=\{X\subset I : Y\subset X\}$ for some $Y\subset I$. But then clearly $\cap D = Y$, which must be in D since $Y\subset Y$.
\end{proof}

\begin{defn}
A filter D is said to be an \textbf{ultrafilter} if for all $X\in \mathcal{P}(I)$, $X\in D$ iff $(I\setminus X)\notin D$.
\end{defn}

Filters provide a notion of largeness—in particular, those sets belonging to a filter can typically be thought of as ``large" sets. (An exception might arise in the case of a principal filter, where even certain finite sets containing some particular set Y are considered large.) It is also worth noting that an ultrafilter can be conceptualized as a binary measure on an index set: it assigns every set either measure 0 or measure 1. Ultrafilters can also be characterized as maximal proper filters (for a proof see [1]), and principal ultrafilters are generated not just by a single subset (as in their definition), but indeed by singleton sets.

\begin{prop}
D is a principal ultrafilter over I if and only if $D = \{X\in \mathcal{P}(I) : i\in X\}$ for some $i\in I$.
\end{prop}

\begin{proof}
First suppose D is a principal ultrafilter. Then $\exists Y\subset I$ s.t. $D=\{X\subset I : Y\subset X\}$. We claim that $Y = \cap D$ is a singleton set $\{i\}$ such that in fact $D = \{X\in \mathcal{P}(I) : i\in X\}$. Suppose not. Then Y is the disjoint union of two sets A and B, both of which are too small to be in D since Y is $\subset $-minimal in D. Now since D is an ultrafilter, $I\setminus A$ and $I\setminus B$ are in D. Further, since D is a filter and closed under finite intersections, $Y\cap (I\setminus A) \in D$. But $Y\cap (I\setminus A) = B$, which is a contradiction since $B\notin D$. The other direction is trivial as any element $i\in I$ corresponds to a singleton set $\{i\}\subset I$ which we take to be our $Y\subset I$.
\end{proof}

Without the axiom of choice, it is not even clear that nonprincipal ultrafilters exist. Nevertheless, this is the case, and follows from an easy application of Zorn's lemma. Although we shall prove the Ultrafilter Lemma using Zorn’s lemma, it is worth noting that the existence of nonprincipal ultrafilters is not an equivalent formulation of the axiom of choice (as Zorn’s lemma is.) In fact, the Ultrafilter Lemma is implied by, but does not imply, AC.\footnote{Interestingly, this weaker assumption of the Ultrafilter Lemma is all that is needed to perform the Banach-Tarski decomposition of the sphere.}

\begin{thm} (Ultrafilter Lemma)
Every proper filter D over I can be extended to an ultrafilter over I. In particular, nonprincipal ultrafilters exist. 
\end{thm}

\begin{proof}
Let $P=\{E\subset \mathcal{P}(I): E \text{ is a proper filter over I and } D\subset E\}$. Let $C$ be a nonempty chain of elements of $P$, ordered by set inclusion. Then $\cup C$ is clearly a proper filter containing D and hence is an upper bound in P. Thus, Zorn's lemma applies, so P has a maximal element F. Since F is a maximal proper filter, F is an ultrafilter. Letting our D be, for example, the Fréchet filter over any infinite I, we obtain a nonprincipal ultrafilter.
\end{proof}

The following Theorem 4.6 is of independent interest because it is the mathematical way of stating Arrow’s Impossibility Theorem of social choice theory.

\begin{thm}
(Arrow's Impossibility Theorem) Every ultrafilter over a finite set is principal.
\end{thm}

\begin{proof}
Let I be finite and let D be a filter over I. Since I is finite, $\mathcal{P}(I)$ is finite as well. Since D is a subset of $\mathcal{P}(I)$, D must be finite. Therefore $\cap D\in D$ (since filters are closed under finite intersections). So by Proposition 3.2, D is principal. 
\end{proof}

Somewhat loosely speaking, Arrow’s theorem states that any voting system that satisfies 1) weak Pareto efficiency, and 2) independence of irrelevant alternatives, is in fact a dictatorship. Intuitively, this makes sense in connection to the above theorem when we consider that the property of having \textit{one particular} set be the determiner of membership in the filter (i.e., “majority”) is analogous to the property of being a dictator.

To demonstrate that in stepping into the world of filters we have entered an arena where our size intuitions might fail us, we give one result on ultrafilters that may feel surprising and nontrivial. It can be thought of as a coloring lemma: given an ultrafilter and a function from the index set to itself, either there exists a ``large" set of fixed points or the set of non-fixed points can be 3-colored such that no point maps to a point of the same color. The proof is worth elaborating because it is strongly analogous to the proof of the Schröeder-Bernstein theorem.

\begin{prop}
Let D be an ultrafilter over a set I and let $f: I\to I.$ Then there is a set $X\in D$ such that either $f(i)=i$ for all $i\in X$, or $f(i)\notin X$ for all $i\in X$.
\end{prop}

\begin{proof}
Let $J=\{i\in I : f(i)\neq i\}.$ Then since D is an ultrafilter, either $J\in D$ or $I\setminus J\in D$. If $I\setminus J\in D$ we are done as $I\setminus J$ satisfies the property that $f(i)=i$ for all $i\in I\setminus J$. So we can assume that $J\in D$. If J can be partitioned into 3 sets $X_1, X_2, X_3$ such that for $n\in \{1,2,3\}$, $f(i)\notin X_n$ for all $i\in X_n$ then at least one of $X_1, X_2, X_3$ must be in D because the \textit{dual ideal}\footnote{Every filter admits of a dual ideal, where the definition of an ideal is the inverse of that of a filter. In particular, ideals contain the empty set, are closed under finite unions, and are downwards closed, thus providing a notion of smallness.} of D is closed under finite unions and we already have that $J\in D$. 

We now make the following observations. For every point $i\in I,$ the sequence obtained by repeatedly applying f either ends in a fixed point (in $I\setminus J$), ends in a cycle, or is an infinite non-repeating sequence. We will show it is possible to construct $X_1, X_2, X_3$ with the desired property. To begin our construction, we initially let $X_1 = \{i\in J : f(i)\in I\setminus J\}$. Recursively, let $A_1 = \{i\in J : f(i)\in X_1\}$ and in general let $A_k = \{i\in J : f(i)\in A_{k-1}\}$. Now we initially let $X_2 = \{\cup A_k : k\bmod 2 = 0\}$ and $X_3 = \{\cup A_k : k\bmod2 = 1 \}$. This gives us a partial partition with the desired property because by definition f maps elements of $X_2$ to $X_3$, elements of $X_3$ to $X_2$ or to $X_1$ and elements of $X_1$ to $I\setminus J \not\subset J$. However, we have not yet exhausted J, as there may be elements of J whose paths under f end in a cycle or constitute an infinite non-repeating sequence. But for such remaining elements, we can do the following: Consider the set $C_1 = \{i\in J : \exists n\in \mathbf{N} f^n(i)=i\}.$ Note that $C_1$ can be partitioned into sets of cycles $D_i$. We can arbitrarily assign every third element in each $D_i$ to one of $X_1, X_2, $ and $X_3$ (unless the cycle $D_i$ has length equal to 1 (mod 3), in which case we simply ensure in one of many possible ways that no adjacent elements of the cycle are assigned to the same $X_n$.) Now for each $D_i,$ we can pick an arbitrary element and assign every element in its fiber $F_{i1}$ to the $X_n$ that at least one element of its fiber is already assigned to (because it is in $D_i$). Now recursively, let $F_{ik} = \{i\in J : f(i)\in F_{i,k-1}\}$. Since each $F_{ik}$ contains at least one element that has already been assigned to an $X_n$ (because each $F_{ik}$ contains an element in $D_i$), we can assign all the elements of $F_{ik}$ to the same $X_n$.

The only remaining possibility is that an element of J has a path under f that constitutes an infinite non-repeating sequence. So we consider the set $B = \{i\in J : \neg \exists n, m\in \mathbf{N} \text{ s.t. } n\neq m\wedge f^n(i)=f^m(i)\}$.

This set can be uniquely maximally partitioned into disjoint infinite sets $T_i$ (for ``tree") that are closed under both f and $f^{-1}$ and such that $\forall x\in T_i, \exists y\in T_i$ s.t. $f(y)=x$. For each $T_i$, let $x_i$ be the unique element such that for all $n\in \mathbf{N}$, $f^n(x_i)$ has a fiber that is a singleton set. For all $n\in \mathbf{Z}$, assign $f^n(x_i)$ (or all the elements of $f^n(x_i)$ when n is negative), to $X_2$ if n is even and $X_3$ for n odd. This gives the desired partition.
\end{proof}

Having built up some experience with ultrafilters, we can now provide the theory leading up to the definition of an ultrapower.

\begin{defn}{Ultraproducts and Ultrapowers}

\begin{itemize}
    \item Let I be a nonempty set, D a proper filter over I, and for each $i\in I$, $A_i$ a nonempty set.
        \item Let $C = \prod_{i\in I} A_i$, the Cartesian product of these sets.
        \item For any two functions $f, g\in C$, we say f and g are D-equivalent, denoted $f =_D g$, iff $\{i\in I : f(i) = g(i)\}\in D$.
        \item Note that $=_D$ is an equivalence relation.
        \item The \textbf{reduced product of $A_i$ modulo D}, denoted $\prod_D A_i$, is the set of all equivalence classes of $=_D$.
        \item So, $\prod_D A_i = \{f_D : f\in C\}$, where $f_D = \{g\in C : f =_D g\}$.
        \item If all the factors $A_i$ are the same, we call this a \textbf{reduced power}.
        \item If D is not just a filter but an ultrafilter, then we have not just a reduced product but an \textbf{ultraproduct} or, in the case that the $A_i$ are all the same, an \textbf{ultrapower}.
        \end{itemize}
\end{defn}

The following theorem is known as the Fundamental Theorem of Ultraproducts. See [1] for a proof.

\begin{thm} (Łoś)
 Let B be the ultraproduct $\prod_D \mathfrak{U}_i$ with D a filter and the $\mathfrak{U}_i$ models indexed by I. Then for any sentence $\varphi$ of our language $\mathcal{L}$, $B\models \varphi$ if and only if $\{i\in I : \mathfrak{U}_i\models \varphi\}\in D$.
\end{thm}

In other words, ultraproducts are ``truth-preserving"—if a statement of our first-order language holds in ``most" of the factor models, it holds in the ultraproduct as well. Note that the statement of truth-preservation is even more direct in the case of an ultrapower because there is only one factor model.

We now give one application of ultrapowers: we will construct a nonstandard model of the natural numbers.
\newline

\begin{exmp}
\item Let $N^*$ be a nontrivial ultrapower (i.e., relative to a nonprincipal ultrafilter) of $\omega$ copies of the natural numbers. We shall call this the hypernatural numbers. Recall that the elements of $N^*$ are actually functions. By the Fundamental Theorem of Ultraproducts, all first-order statements that are true of the natural numbers are true of the hypernaturals. But, the hypernaturals have some weird properties.
        \item For instance, you can have an infinite decreasing sequence (of functions) in $N^*$:
        \begin{itemize}
            \item $1 2 3 4 5 6 7 8 ...$
            \item $1 1 2 3 4 5 6 7 ...$
            \item $1 1 1 2 3 4 5 6 ...$
            \item $1 1 1 1 2 3 4 5 ...$
            \item $1 1 1 1 1 2 3 4 ...$
            \item $1 1 1 1 1 1 2 3 ...$
            \end{itemize}

Each row (function) of the above table is less than the row directly above it on a cofinite (hence ``large") set of elements. It may seem that this violates Łoś's theorem, but in fact it reveals something else: the property ``there are no infinite decreasing sequences" of the natural numbers is not expressible in first-order language. The situation is similar when taking ultrapowers of the real numbers to form the ``hyperreals."\footnote{The hyperreals are a much more commonly studied object in a branch of mathematics known as ``nonstandard analysis" or ``Robinsonian analysis" launched by Abraham Robinson.} The Archimedean property of the reals actually fails for the hyperreals, but this is only because the Archimedean property is not actually expressible as a first-order statement either.
\end{exmp}
Furthermore, we see by a diagonal argument that $N^*$ is in fact uncountable—hence, very nonstandard!

\begin{prop}
Every countable subset of $N^*$ has an upper bound in $N^*$. Thus, $N^*$ is uncountable.
\end{prop}

\begin{proof}
Let D be our nonprincipal ultrafilter and let $C=\{f_{1_D}, f_{2_D}, f_{3_D}, ...\}$ be a countable subset of $N^* = \Pi_D A_i$ where our index set I is $\omega$ and all the factors $A_i = \mathbf{N}$. We can think of this as a countable set of sequences of natural numbers. We construct an upper bound $f_{U_D}$ as the equivalence class of the function $f_U : \omega \to \omega$ which we define by the diagonalization method as follows: let $f_U(j)=max_{i\leq j}{f_i(j)}+1$. Then $f_U$ is greater than each $f_i$ on all but $i-1$ coordinates (that is, on all but a finite set of coordinates.) Since D is nonprincipal, the cofinite sets must be in D. So $f_{U_D}$ is indeed an upper bound for C.
\end{proof}

On the other hand, it is useful to verify by example that our statement that ``everything that is true of the natural numbers is true of the hypernaturals" for a statement that is expressible in first-order logic. We choose that statement ``every natural number is either zero or the successor of some other natural number" as our example.

If it is not the case that a ``number" (function) is zero in $N^*$, then the set of coordinates on which that function equals 0 is not in our ultrafilter D. But then its complement, the set of coordinates on which it does not equal 0, is in D, since D is an ultrafilter. Any coordinate that does not equal 0 is seen to be a successor since we can subtract 1 and obtain a predecessor (in other words we are simply invoking that ``every natural number is either 0 or the successor of a natural number" is indeed true in $\mathbf{N}$). So the set of coordinates that are successors (equivalently, ``not-0") is in D, and thus ``every number is either zero or a successor" holds in $N^*$ as well (because we showed that non-zero implies successor.) However, note the importance of having an \textit{ultra}filter: if we were to construct $N^*$ using some D that is a non-ultra filter, then if it is not the case that a ``number" (function) is zero, the set of coordinates on which it is 0 would not be in D, but the complement of this set need not be in D either. That is, the set of coordinates that are not 0 (successors, in the ordinary natural number sense) is not necessarily ``large." So we cannot conclude from a ``number" being non-zero that it is a successor. There could be numbers in such a model that are neither zero nor successors.

\section{Measurable cardinals}

After a lengthy digression, we are now ready to return to the definition of a measurable cardinal.

\begin{defn}
An ultrafilter is \textbf{$\alpha$-complete} if it is closed under the intersection of strictly less than $\alpha$-many sets.
\end{defn}

\begin{defn}
A cardinal $\alpha$ is said to be \textbf{measurable} if there exists a nonprincipal, $\alpha$-complete ultrafilter over $\alpha$.
\end{defn}

Technically, $\omega$ is measurable, and this conforms with our intutions that measurable cardinals should be ``much larger" than anything below them. However, we shall only really care about uncountable measurable cardinals and will henceforth use the term ``measurable" with that meaning. It can be shown that measurable cardinals are indeed inaccessible, but it was an open question for over three decades whether the least inaccessible cardinal might in fact be measurable. Hanf eventually showed this was not the case. In fact, we have the following, which suggests just how large measurable cardinals are:

\begin{thm}
If $\kappa$ is an (uncountable) measurable cardinal, then $\kappa$ is the $\kappa$-th inaccessible cardinal.
\end{thm}

Set theorists have strong intuitions for the consistency of measurable cardinals (for more on this, see section 8.) Thus, the following monumental result may be seen as a blow to our prospects for simplicity in our characterization of the universe of sets. In promoting the importance of set theory, Hilbert once famously quipped that no one could expel us from Cantor's paradise. Nevertheless, Scott swiftly expels us from Gödel's paradise:

\begin{thm}
(Scott) If there exists a measurable cardinal, then $V\neq L$.
\end{thm}

We defer the proof until the next section.

\section{The general template for large cardinals}

``Large" large cardinals are those at the level of measurability and beyond (i.e., those not compatible with V=L). By contrast, ``small" large cardinals are large cardinals for which it is not known to be inconsistent that they could exist in the constructible universe. There is a general template that characterizes essentially all of these ``large" large cardinals, which we will see now.

\begin{defn}
An \textbf{inner model} (of ZFC) is a proper class M that is a transitive $\in$-model of ZFC with $ON\subseteq M$.
\end{defn}

Gödel's constructible universe L is the canonical inner model and is in fact the minimal inner model in the sense that $L\subseteq M$ for any inner model M.

\begin{defn}
For structures $M_0$ and $M_1$ of a language L, an injective function $j: M_0\to M_1$ is an \textbf{elementary embedding} if for any formula $\phi(v_1,...v_n)$ of L and $x_1,...x_n\in M_0$, $M_0\models \phi[x_1,...x_n]$ iff $M_1\models \phi[j(x_1),...,j(x_n)]$.
\end{defn}

Elementary embeddings can be viewed as truth preserving mappings from V to some transitive class M, which thus turns out to be an inner model. We only consider elementary embeddings that are not the identity map, and this assumption will henceforth not be made explicit. It is easy to show that if all of the ordinals are fixed points under j, then j is in fact the identity map. Contrapositively, for any nontrivial elementary embedding, there must exist some ordinal not fixed by it, in light of which the following definition makes sense:

\begin{defn}
For $j:N\to M$ an elementary embedding where $N, M$ are inner models, the \textbf{critical point} of j, $crit(j)$, is the least ordinal $\delta$ moved by j.
\end{defn}

We can use this to give an alternative second-order characterization of measurability, which we could have taken to be the definition of a measurable cardinal except that it would not conform with the concept's historical origins.

\begin{prop}
A cardinal $\kappa$ is a measurable cardinal iff there exists a transitive class M and an elementary embedding $j: V\to M$ such that $\kappa$ is the critical point of j.
\end{prop}

%\begin{proof}
%We follow [8]. First suppose that $\kappa$ is a measurable cardinal. Then we will show that $\kappa$ is the critical point of an elementary embedding $j:V\to M$ where M is (isomorphic to) an ultrapower of V with respect to an ultrafilter U witnessing the measurability of $\kappa$. Suppose for contradiction that there exists an $\alpha < \kappa$ such that 
%\end{proof}
The proof of the above operates in the forward direction by showing that if $\kappa$ is a measurable cardinal, then $\kappa$ is the critical point of an elementary embedding $j:V\to M$ where M is (isomorphic to) an ultrapower of V with respect to an ultrafilter U witnessing the measurability of $\kappa$. In the other direction, given an elementary embedding $j$ with critical point $\kappa$, we construct an ultrafilter by collecting those subsets of $\kappa$ whose image under $j$ contains $\kappa$ as an element, and then show that this ultrafilter is nonprincipal and $\kappa$-complete over $\kappa$. We omit the details here, turning instead to now giving a proof of Scott's theorem.

\begin{thm}
(Scott) If there exists a measurable cardinal, then $V\neq L$.
\end{thm}
\begin{proof}
Suppose V=L. Let $\alpha$ denote the least measurable cardinal and let $j:V\to M$ witness the measurability of $\alpha$. Then $M\subseteq V$ and since $V=L$ we have $M\subseteq L$. But we know L is the minimal inner model, so in fact $M=L$. Now since j is an elementary embedding, hence truth-preserving for first-order statements, $(j(\alpha)$ is the least measurable cardinal)$^M$. (That is, the model M ``thinks" that $j(\alpha)$ is the least measurable cardinal.) But  $V=M=L$ so $j(\alpha)$ is the least measurable cardinal in V as well. However $\alpha$ is the least ordinal moved by j, so in particular $j(\alpha) > \alpha,$ and thus we have contradicted the assumption that $\alpha$ was the least measurable cardinal.
\end{proof}

Notice that the issue here (perhaps suppressed by the bird's eye form in which we have given the proof) is not that the measurable cardinal fails to exist in L. In passing from V to L, we keep all the same ordinals, so the set that was a measurable cardinal is still there, but we excise enough of the universe such that we lose its witnessing ultrafilter. Thus, it is no longer measurable—we cannot ``appreciate its largeness" anymore.

\subsection{Beyond measurability}

In general, the definitions of large cardinal properties stronger than measurabililty follow this same template involving critical points of elementary embeddings $j:V\to M$, but impose additional closure conditions on the target model, essentially requiring M to be ``closer" to V.

We first give an example by way of \textit{supercompact} cardinals, briefly exhibit some even stronger notions, and finally prove a limiting case that arises with \textit{Reinhardt} cardinals.

\begin{notn}
$^yx$ denotes the collection of functions from y to x.
\end{notn}

\begin{defn}
A cardinal $\kappa$ is \textbf{$\gamma$-supercompact} if there is a $j: V\to M$ such that $crit(j)=\kappa$ and $^\gamma M\subseteq M$.
\end{defn}

That is, our closure condition here requires $\gamma$-length sequences of elements of M to still be elements of M. In particular this implies that $H_{\gamma^+}\subseteq M$.\footnote{$H_{\gamma^+}$ denotes the set of all sets hereditarily of size strictly less than the cardinal successor of $\gamma$.}

\begin{defn}
A cardinal $\kappa$ is \textbf{supercompact} if $\kappa$ is $\gamma$-supercompact for every $\gamma\geq \kappa$.
\end{defn}

Note that for a cardinal $\kappa$, $\kappa$-supercompactness is equivalent to $\kappa$ being measurable. In fact, supercompactness can be seen as generalizing the notion of measurability.

\begin{prop}
If $\kappa$ is $2^\kappa$-supercompact, then there exists a (normal) ultrafilter U over $\kappa$ such that $\{\alpha\leq \kappa : \alpha$ is measureable$\} \in U$. In particular, $\kappa$ is the $\kappa$-th measurable cardinal.
\end{prop}

Supercompactness turns out to be the crucially important notion for the current direction of inner model theory, as we shall see later.

The general template in terms of elementary embeddings allows us to define a flurry of other large cardinal notions, though we shall skimp on their motivations as they lie beyond the scope of this paper.

\begin{defn}
$\kappa$ is \textbf{$\eta$-extendible} if for some ordinal $\lambda$ there is a $j:V_{\kappa+\eta}\to V_\lambda$, where $\kappa=crit(j)$. $\kappa$ is \textbf{extendible} if $\kappa$ is $\eta$-extendible for every nonzero ordinal $\eta$.
\end{defn}

\begin{defn}
$\kappa$ is \textbf{huge} if there is a $j:V\to M$ such that $\kappa=crit(j)$ and $^{j(\kappa)} M\subset M$.
\end{defn}

There are a number of closely related variants of huge cardinals; for instance, an \textit{almost huge} cardinal is one for which the closure condition instead applies to all sequences of length strictly less than $j(\kappa)$. Huge cardinals are all of their cousins are of much greater relative consistency strength than extendible cardinals. A large cardinal property intermediate between extendible and huge, known as \textit{Vopênka's Principle}, has a curious history. It can be stated as follows: ``for any proper class of structures for the same language, there is one that is elementarily embeddable into another." This cannot be formulated as a single sentence of ZFC because it involves quantification over proper classes, which is why we call it a ``principle." Nevertheless, loosely speaking, we can define a \textit{Vopênka cardinal} to be a cardinal $\kappa$ for which Vopênka's principle holds in $V_\kappa$. Vopênka originally formulated the principle as a joke, believing it to be inconsistent with ZFC. However, before he could publish his inconsistency result, he discovered a flaw in the proof.

On the relationship between the above notions, every almost huge cardinal is a Vopênka cardinal, Vopênka's Principle implies the existence of extendible cardinals, and all extendible cardinals are supercompact.

\subsection{Reinhardt cardinals and Kunen's inconsistency}

It may be natural to ask what the ``ultimate" closure conditions we could impose in defining a large cardinal would be. Reinhardt came up with the following, the obvious idea being that we let M be V itself.

\begin{defn}
A cardinal $\kappa$ is a \textbf{Reinhardt cardinal} if there exists an elementary embedding $j: V\to V$ such that $\kappa=crit(j).$
\end{defn}

Unfortunately, Kunen subsequently showed that the specter of inconsistency finally comes crashing down on us.

\begin{thm}
(Kunen) (ZFC) There are no Reinhardt cardinals. (There is no elementary embedding $j: V\to V$.)
\end{thm}

\begin{notn}
$f``x$ for x a set denotes the image of f when restricted to x; that is, $\{f(y) : y\in x\}$. This notation is necessary in our context because f may be defined on both x and on elements of x.
\end{notn}

\begin{notn}
For $\lambda, \kappa$ cardinals, $[\lambda]^\kappa$ denotes the subsets of $\lambda$ of cardinality $\kappa$.
\end{notn}

\begin{defn}
Let x be a set of ordinals. A function f is $\omega$-Jónsson for x if $f: [x]^\omega\to x$ and for any $y\subseteq x$ with $|y| = |x|$, $f``[y]^\omega = x$.
\end{defn}

That is, an $\omega$-Jónsson may be thought of as a function that colors the $\omega$-sized subsets of x such that every full cardinality subset y of x contains an $\omega$-sized subset of every color.

\begin{lem}
(Erdós-Hajnal) For any $\lambda$, there exists a function that is $\omega$-Jónsson for $\lambda$.
\end{lem}

We shall not prove this lemma but merely note that it makes use of AC and use it in the following proof of Kunen's Inconsistency theorem.

\begin{proof}
Let $j:V\to V$ be an elementary embedding. Note that for elementary embeddings $j:V\to M$ in general, $j$ is not onto. That is we embed the universe within a subset of the target model $M$. However, the requirement of truth-preservation is still relative to all of $M$; that is, all the same formulas must hold in the image of $j$ when we quantify over all of $M$. Let $\kappa=crit(j)$. We can define the \textit{critical sequence} of $j$ to be the sequence obtained by iterating $j$ on $\kappa$. If we take the supremum of the critical sequence, we obtain a fixed point of $j$, which we call $\lambda$. That is, for $\lambda = sup(\{j^n(\kappa) : n\in\omega\})$, we have by elementary equivalence that $j(\lambda) = j(sup(\{j^n(\kappa) : n\in\omega\})) = sup(\{j^{n+1}(\kappa) : n\in\omega\}) = sup(\{j^n(\kappa) : n\in\omega\}) = \lambda$, where the second equality holds because $j$ is continuous at limit ordinals.

So, we have a fixed point $\lambda$ that is above the first non-fixed point $\kappa$. Since $\lambda$ is a fixed point and $j$ is increasing, $j``\lambda \subset \lambda$. We will show that $j``\lambda \notin M$. But $j$ is a (proper class) function and thus for definable $j$, $j``\lambda$ is a set (in V) by the Replacement Axiom, and for arbitrary $j$, $j``\lambda$ is a set by an analogue of the Replacement Axiom, which we must assume in one of various ways.\footnote{See the ensuing metamathematical discussion following the proof. The simplest way to do this, though not necessarily the most natural or economical in terms of consistency strength, would simply be to augment  ZFC by Replacement for elementary embeddings and then work in the resulting theory, which we may call $ZFC(j)$.} Therefore it will follow that $M\neq V$.

To show that $j``\lambda \notin M$, we exploit the $\omega$-Jónsson property. Let $f$ be $\omega$-Jónsson for $\lambda$. The $\omega$-Jónsson property is preserved by $j$ since $j$ is an elementary embedding, so $j(f)$ is $\omega$-Jónsson for $j(\lambda)$, which is $\lambda$. Furthermore, since $j$ is injective, $j``\lambda$ is a $\lambda$-sized subset of $\lambda$. So, in particular, $j``\lambda$ is a full-cardinality subset of $\lambda$, and thus the hypotheses of the Erdós-Hajnal lemma are satisfied. Thus, $j(f)``[j``\lambda]^{\omega}$ ought to be $\lambda$. But instead we will show that  $j(f)``[j``\lambda]^{\omega} \subseteq j``\lambda$, which is certainly strictly smaller than $\lambda$ since in particular $\kappa$, as the critical point, cannot be in the image of $j$, and hence $\kappa \in \lambda\setminus j``\lambda$.

Thus, let $s\in [j``\lambda]^\omega$. Then we can take a preimage $t\in[\lambda]^\omega$ such that $j(t)=s$. Now $j(f)(s)=j(f)j(t)$, which by the elementary equivalence of $j$ is equal to $j(f(t))$. But recalling that the codomain of f is $\lambda$, $j(f(t))\in j``\lambda$, as desired.

\end{proof}

To summarize alternatively, it can be shown (by taking $\omega$ iterates from any given point), that any elementary embedding $j: V\to V$ has arbitrarily large fixed points. But below any fixed point, as a contingency of the $\omega$-Jónsson property, $j$ must be onto. Therefore $j$ must be onto V, and thus any purported elementary embedding $j$ must in fact be trivial.

As a metamathematical aside, observe that since j is a class and quantification over classes cannot be formalized in ZFC, this proof of Kunen's theorem, which makes a statement about \textit{all} such j, cannot be regarded as a single proof in ZFC but rather only as an infinite schema of proofs, one for each j. Hamkins points out that, formulated as such in ZFC, this only rules out \textit{definable} elementary embeddings from V to V. But if this were our only goal, a simpler argument that does not depend on infinitary combinatorics nor even the axiom of choice is available. Roughly, it states that the concept of being a Reinhardt cardinal cannot be expressible in first-order language because if it were, then letting $\kappa$ denote the least Reinhardt cardinal (i.e., minimal critical point) relative to the set of possible parameters defining a nontrivial embedding, $j(\kappa)$ must also satisfy this definition, contradicting $\kappa < j(\kappa)$.\footnote{Understood at approximately this level, this result was for many years a ``folk theorem" of set theory until Suzuki published it formally in 1999 by showing rigorously that this definition is first-order expressible.} For this reason, Kunen intended for his proof to be formalized in Kelley-Morse (KM) set theory, which unlike ZFC allows for quantification over classes. However, KM is slightly stronger than ZFC, as it is capable of proving $Con(ZFC), Con(ZFC + Con(ZFC)),$ and so on. Thus, it would actually be most desirable to formalize the Kunen inconsistency result in von Neumann-Bernays-Gödel (NBG) set theory, which is similar to KM but equiconsistent with ZFC because while it allows for classes, it differs from KM in the respect that it only allows classes to be defined by formulas with quantifiers that range over sets. With a few more technical difficulties, Hamkins shows that this formalization in NBG is possible—for more on this, see [6]. Thus, viewed in full generality, Kunen's inconsistency theorem can indeed be seen to rule out any nontrivial embedding from the set-theoretic universe to itself. It merits mention that still other proofs of Kunen's inconsistency theorem are possible; for more see [9].

It may appear that we have walked into a faux pas by continuing to refer to Reinhardt's concept as a ``cardinal" at all if it is inconsistent. Perhaps the reason we do is that is it still an important open question whether Reinhardt cardinals are inconsistent with ZF alone (the Erdös-Hajnal lemma makes crucial use of AC). Along these same lines, it is also worth noting that despite the inconsistency of Reinhardt cardinals, even stronger large cardinal notions exist: for instance, \textbf{Berkeley cardinals}, proposed by Woodin during his tenure at the university of the same name, are incompatible even with the axiom of countable choice (but potentially are consistent with ZF).\footnote{Question: What do you call someone who roots for both Cal and Stanford? Answer: A Berkeley cardinal; they're wildly inconsistent.}

\section{Conceptualizing the large cardinal hierarchy}

Much of the confusion about the large cardinal hierarchy stems from rampant less-than-careful wording. There are, I believe, at least three distinct important notions that must be clarified when speaking of the large cardinal hierarchy: relative consistency strength, direct implication, and questions of cardinality.\footnote{It should not go without mention that there is also the important question of \textit{interpretability strength}, although this is beyond the scope of this paper.} However, the relationship between these three concepts is at best murky. We are primed by the word “large” to think of cardinality, but first of all, it is not even clear what the proper notion of cardinality would be in the context of the large cardinal hierarchy. This is because, above a large cardinal of type A, there are many cardinals that are not of type A (and many that are not even large cardinals at all). Thus, perhaps the main useful question we can ask with regard to cardinality is where the least cardinal of type A falls. Of the three notions, relative consistently strength is probably the most important and most frequently discussed. (If not specified, one should assume that comparative statements about large cardinals are referring to consistency strength.) We say that a large cardinal of type A is of greater relative consistency strength than a large cardinal of type B if Con(ZFC + ``there exists a cardinal of type A") $\vdash$ Con(ZFC + ``there exists a cardinal of type B"). Statements about relative consistency strength can occasionally be strengthened to statements of direct implication, which take the form ``every large cardinal of type A is a large cardinal of type B." To give another example of direct implication, it shall be useful to define strongly compact and weakly compact cardinals. However, to do so, we must first introduce a new logical concept.

Roughly speaking, we introduce the notion of an \textit{infinitary language} $L_{\lambda\mu}$ where we allow conjunctions and disjunctions of $<\lambda$ formulas and universal and existential quantifications of $<\mu$ variables. Thus, $L_{\omega\omega}$ is just the language of our ordinary first-order logic, for which the Compactness Theorem holds (if all finite subsets of a set of sentences $\Sigma$ have a model, then $\Sigma$ has a model.) Generalizing this notion we get the following:

\begin{defn}
$\kappa$ is strongly compact iff: any collection of $L_{\kappa\kappa}$ sentences, if $\kappa$-satisfiable, is satisfiable.
\end{defn}

\begin{defn}
$\kappa$ is weakly compact iff: any collection of $L_{\kappa\kappa}$ sentences using at most $\kappa$ non-logical symbols, if $\kappa$-satisfiable, is satisfiable.
\end{defn}

\begin{thm}
Every strongly compact cardinal is measurable.
\end{thm}

\begin{thm}
Every measurable cardinal is weakly compact.
\end{thm}

Furthermore, any measurable cardinal has many weakly compact cardinals below it. Magidor showed that the size of the least strongly compact cardinal is independent of ZFC, but it need not be greater than the least measurable cardinal.

We give one more definition, that of a Woodin cardinal, albeit with limited motivation at this time. However, Woodin cardinals are crucially important and among the most prominent of all large cardinal hypotheses. We will see one reason for this in the section on determinacy.

\begin{defn}
A cardinal $\kappa$ is \textbf{Woodin} if for any $f\in$ $^\kappa \kappa$, there is an $\alpha < \kappa$ with $f''\alpha\subseteq \alpha$ and a $j: V\to M$ with crit(j) $= \alpha$ such that $V_{j(f)(\alpha)}\subseteq M$.
\end{defn}

Woodin cardinals are much stronger than measurable cardinals in terms of consistency strength. However, a Woodin cardinal need not be measurable, and the least Woodin cardinal is known to not even be weakly compact. The relationship between huge cardinals and supercompact cardinals, which we defined earlier, is similar: huge cardinals are stronger in consistency strength, but assuming both exist, the least huge cardinal is of lesser cardinality than the least supercompact cardinal.

Seeming pathologies like this pervade the large cardinal hierarchy, which is why diagrams like the one below are not particularly useful without further clarification.

\begin{figure}
    \centering
    \includegraphics[width=12cm]{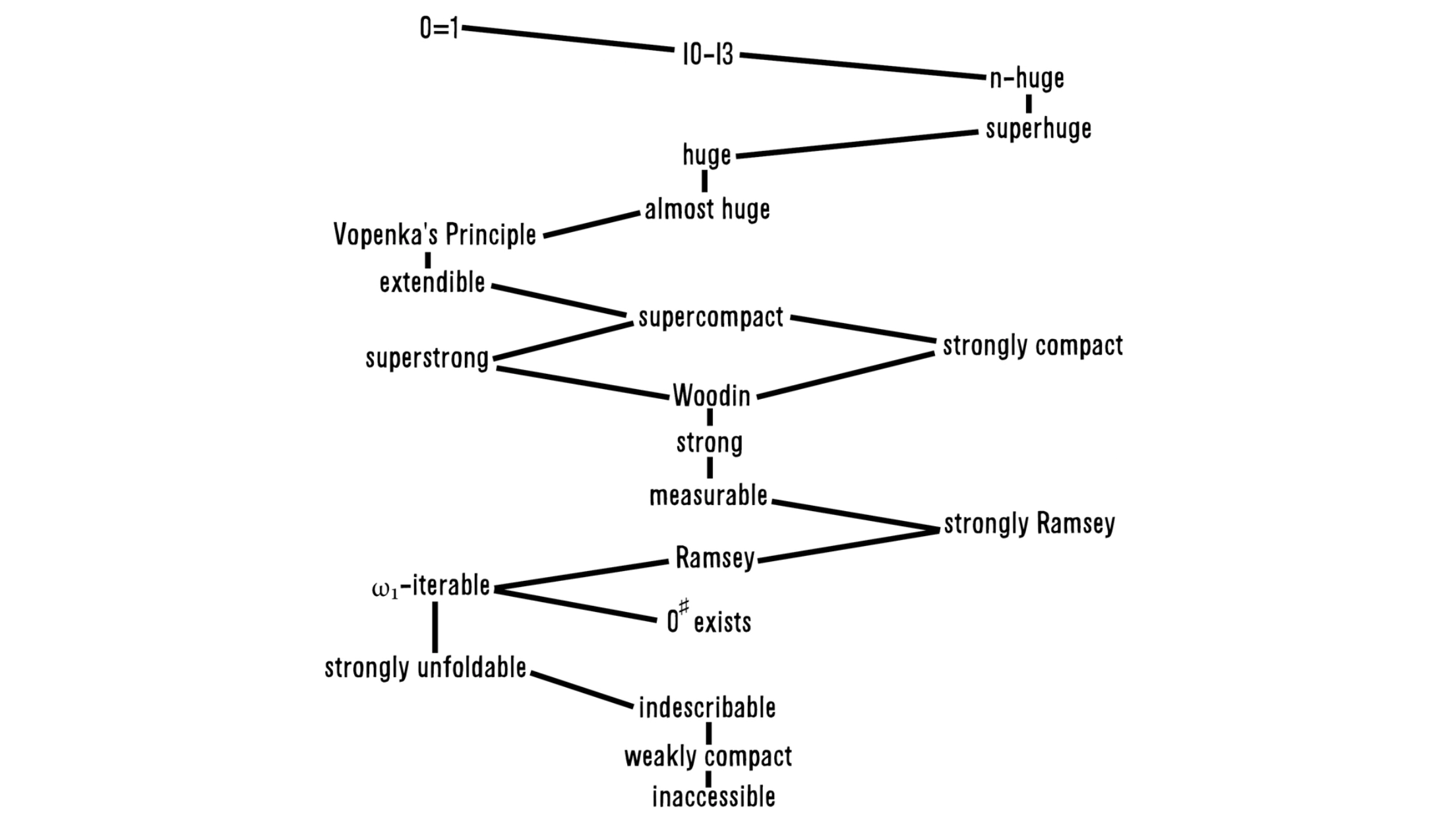}
    \caption{Large Cardinal Relationships}
    \label{Figure 6.7}
\end{figure}

Figure 1 is best understood as mapping out the relationships in terms of relative consistency strength. Note that by the \textit{principle of explosion}, anything follows from a contradiction, which explains the presence of the statement $0=1$ at the top. However, many questions of direct implication and cardinality are suppressed by the diagram, and even the known relationships in terms of consistency strength are not made entirely clear. For instance, supercompactness is at least as strong of an assumption as strong compactness, but it is an open question whether the two are in fact equiconsistent (and many set theorists conjecture this to be the case). We shall not give the definition of a \textit{strong cardinal}, but the reader should know that one salient property of strong cardinals is that if it is consistent for one to exist, then it is consistent for there to be no inaccessibles above some strong cardinal. The relationship between strongness and strong compactness is quite unclear.

One may ask what the strongest large cardinal axiom is.\footnote{With apologies for the many different uses of the word ``strong." Here, of course, we are referring to strength in terms of relative consistency.} Even with the presence of limiting cases like Kunen's inconsistency, it seems as though we may always come up with stronger and stronger large cardinal hypotheses. However, at present the strongest large cardinal hypotheses that have been widely considered are the \textit{rank-into-rank} axioms, the strongest of which is known as Axiom $I_0$. This is essentially the strongest large cardinal axiom not known to be inconsistent with the axiom of choice.

\begin{defn}
\textbf{Axiom $\mathbf{I_0}$} asserts that there exists a (nontrivial) elementary embedding $j: L(V_{\delta+1})\to L(V_{\delta+1})$ such that $\kappa$ is the critical point of j and $\kappa < \delta$.
\end{defn}

\section{Determinacy}

Having just taken care to not (knowingly) contradict AC, we now enter the world of determinacy, or infinite game theory, where AC is the first thing to go out the window.

 We imagine a two-player infinite game $G_X(A)$ for $A\subseteq$ $^\omega X$, where player I and player II alternate between picking elements of X. In the simplest case, let $X=\omega$, so each player picks natural numbers. If the resulting infinite sequence of natural numbers is in $A\subseteq$ $^\omega\omega$, player I wins; otherwise player 2 wins. If either player has a winning strategy, then we say the set A of real numbers\footnote{Technically A consists of infinite sequences of natural numbers, which can be thought of as real numbers and allows us to circumvent the more involved task of precisely defining a real number.} is \textbf{determined}. The following is basic; for a proof see [9].
 
 \begin{prop}
        If $A\subseteq$ $^\omega \omega$ is either open or closed, then $G_X(A)$ is determined.
 \end{prop}
 
\begin{thm}
(ZFC) There exists a set of reals which is not determined.
\end{thm}
\begin{proof}
Note that a strategy is a function from information sets to actions, and as there are only countably many of each for games of the form $G_\omega(A)$, each player has only ${\aleph_0}^{\aleph_0}=2^{\aleph_0}$ many strategies. Using the Axiom of Choice, we can well-order the possible strategies for player I as $\langle \sigma_\alpha | \alpha < 2^{\aleph_0}\rangle$ and the strategies for player II as $\langle \tau_\alpha | \alpha < 2^{\aleph_0}\rangle$. We now use a back-and-forth construction to assign certain $a_\alpha, b_\alpha\in$ $^\omega \omega$ (i.e., real numbers) to each player's winning set. To do so we employ transfinite recursion: having chosen reals $a_\beta$ and $b_\beta$ for $\beta < \alpha$, alternate between choosing $b_\alpha$ such that $b_\alpha = \sigma_\alpha * y$ for some strategy y (so that Player II obtains a way to counter against this strategy being played by Player I) and choosing $a_\alpha$ such that $a_\alpha = z * \tau_\alpha$ for some strategy z (so that Player I obtains a way to counter against this strategy being played by Player II). In each case we make sure that the real number we assign has not already been assigned, so we impose the additional requirement that $b_\alpha\notin \{a_\beta | \beta < \alpha\}$ and $a_\alpha\notin \{b_\beta | \beta < \alpha\}$. We can do this because $|\{\sigma_\alpha * y | y\in ^\omega$$\omega\}| = |\{z * \tau_\alpha | z\in ^\omega$$\omega\}| = 2^{\aleph_0}$; that is, having only made at most countably many commitments at any stage, there are always $2^{\aleph_0}$ many positions each player can still reach so there always exists some such y or z.  Thus $A=\{a_\alpha | \alpha < 2^{\aleph_0}\}$ and $B=\{b_\alpha | \alpha < 2^{\aleph_0}\}$ are disjoint, and neither player has a winning strategy for $G_\omega(A)$. So A is not determined.

\end{proof}

The \textbf{Axiom of Determinacy (AD)} is the statement that all sets of reals are determined; thus, AD is incompatible with AC. Nevertheless, studying its consequences under the ambient theory ZF can lead to new insights.

Recall that L is the smallest inner model of ZFC containing all the ordinals. Analogously, $L(\mathbb{R})$ is the smallest inner model of ZFC containing both all the ordinals and all the real numbers. Without positing large cardinals, it cannot be shown that $L(\mathbb{R})$ differs from L, but assuming sufficient large cardinal hypotheses, $L(\mathbb{R})$ is a model of AD, not AC. The following makes this precise:

\begin{thm}
Suppose there are infinitely many Woodin cardinals with a measurable cardinal above them. Then $AD^{L(\mathbb{R})}$.
\end{thm}

We will skip much of the theory of determinacy and fast forward to a fundamental, surprising result that will motivate some of our philosophical discussion in the next section.

\begin{thm}
(Woodin) The following theories are equiconsistent:
\begin{enumerate}
    \item ZFC + there exist infinitely many Woodin cardinals.
    \item ZF + AD
\end{enumerate}
\end{thm}

\section{Belief in the consistency of large cardinals}

Why do set theorists believe that large cardinals are consistent? As we noted in the introduction, Gödel's Second Incompleteness theorem poses an insurmountable obstacle for the prospects of \textit{proving} (from ZFC) the consistency of any large cardinal property. Thus, if large cardinals are in fact consistent, this belief must be justified on the basis of philosophical argumentation as opposed to mathematical proof. Some, like Woodin, go so far as to say that the consistency of large cardinals constitutes a new type of mathematical knowledge.\footnote{Woodin is particularly confident in the consistency of the large cardinal bearing his name, predicting that there will be no discovery of a contradiction from the assumption of the existence of infinitely many Woodin cardinals anytime in the next 10,000 years, or indeed ever.} But regardless of one's stance on this, there is no doubt that large cardinals are indispensable tools in evaluating the strength of mathematical theories and are here to stay. Thus, we consider three lines of argumentation for their consistency.

\subsection{The argument from experience}

Some will contend that due to the vast experience set theorists have built up with large cardinals, it is unlikely any of the large cardinal hypotheses in the literature are inconsistent. The argument here is essentially that if there were a contradiction from, say, a measurable cardinal, then we would have discovered it in the decades set theorists have spent studying their properties.\footnote{Famously, Silver was an outlier among set theorists in that he believed measurable cardinals were inconsistent and spent the last two decades of his career at Berkeley attempting to derive a contradiction from their assumption. Other set theorists may have chided him for pursuing this line of investigation seriously, but while of course Silver did not succeed, he made many other important advances in the process.} This argument may seem to be naive, and for the most part, I think it is actually a misconception that this serves as a primary justification for large cardinals. But we perhaps can understand it better by analogy. We all believe that the Peano axioms of arithmetic are consistent because we have a clear picture in our heads of a model of them—namely, the natural numbers. For “small” large cardinals (those weaker than measurables and consistent with V=L), we have another natural model, namely L. That said, it is indisputable that as a mathematical object, L is far more complicated than $\mathbb{N}$, and certainly inner models of large cardinals at the level of measurability and beyond are still more complicated. But when set theorists make the argument from experience, their justification is based at least in part on their greater fluency with such models than ordinary mathematicians.

\subsection{The argument from truth}

Another argument for the consistency of large cardinals comes from the Platonist school of thought. One might think that large cardinals axioms are consistent because in fact they are \textit{true}; that is, large cardinals exist in some Platonic sense. While this may seem like an extreme view, again, it is not so far-fetched to say that our belief in the consistency of Peano arithmetic is due to the fact that we think the natural numbers ``really exist."\footnote{Unless, of course, you are a \textit{finitist} and doubt the existential status of mathematical infinity altogether.} Likewise L is a fairly specific and concrete mathematical object, and it is not terribly unreasonably to suppose it might have a Platonic manifestation. So perhaps large cardinals ``really exist" (or exist in their inner models) in some sense as well. But obviously, this stance requires a considerable leap of faith and is not necessarily the most tenable.\footnote{This is especially true for those large cardinals for which we have not solved the inner model problem.}

\subsection{The argument from structure theory}

The last line of argumentation we consider is what is sometimes referred to as ``structure theory" and can be approximately captured by the following statement of Woodin: ``In many cases, very different lines of investigation have led to problems whose degree of unsolvability is exactly calibrated by a notion of infinity." Structure theory provides the most promising avenue for justifying large cardinal properties, but we shall begin with an argument frequently associated with structure theory that is not actually all that compelling. Some contend that it is a remarkable and compelling empirical observation that all of the large cardinal properties we have elaborated thus far appear to be linearly ordered by consistency strength.\footnote{This is only an empirical observation, not a mathematical fact. In fact, without an agreed upon definition of what constitutes a (natural) large cardinal property, it would appear that prospects for formalizing this claim are limited.} This seems to reflect some kind of deep structure in the realm of mathematical truth and thus might be said to be a point in favor of their consistency, but that would be misguided. For one thing, Hamkins and Koellner both point out that a fundamental problem in the philosophy of mathematics generally is to explain why all ``naturally-arising" mathematical theories have the property of being linearly ordered by consistency strength—in other words, the situation is not unique to large cardinals. Second, Hamkins points out that \textit{in practice}, relative consistency proofs of large cardinal theories generally have the following form: from a model of a large cardinal of type A, we proceed by forcing or inner model methods to construct a model of a cardinal of type B. However, all of these methods preserve arithmetic truth, and for two theories to be incomparable, there must be instances of arithmetic statements that are true in one but not the other. Given this state of affairs, it should be viewed merely as a case of confirmation bias that the large cardinal properties being studied are linearly ordered by consistency strength—our standard methodology insists upon it.\footnote{As a final consideration, Hamkins also challenges the view that the large cardinal hierarchy even is well-ordered by consistency strength by using what he calls ``cautious enumerations" to construct instances of incomparable theories extending ZFC. Whether these theories satisfy the ``naturality" requirement is a loaded philosophical question, but beyond the scope of this paper.}

What are the compelling justifications for large cardinals arising out of structure theory then? The theory of determinacy provides a rich and robust source of results that have been shown to be neatly calibrated by large cardinal assumptions, through results like Theorem 7.4, restated here for the ease of the reader.

\begin{thm}
(Theorem 7.4 restated) The following theories are equiconsistent:
\begin{enumerate}
    \item ZFC + there exist infinitely many Woodin cardinals.
    \item ZF + AD
\end{enumerate}
\end{thm}

Gödel's construction of L established half of the relative consistency of the axiom of choice; in particular that $Con(ZF)\vdash Con(ZFC)$. And since Cohen's method of forcing can be used to show $Con(ZF)\vdash Con(ZF + \neg AC)$, this means that $AC$ and $\neg AC$ are equiconsistent hypotheses. It could have turned out that AD was just another characterization of $\neg AC$. But what the above theorem should be seen as doing is making it clear just how much stronger of a hypothesis AD is than AC. You need AC to contradict AD (with ZF as the ambient theory) \textit{just in case} it is consistent for there to exist infinitely many Woodin cardinals.\footnote{A fun corollary is that if you can prove Theorem 7.2 without invoking AC, you will have proven that the assumption of infinitely many Woodin cardinals is inconsistent.} The concept of a Woodin cardinal did not emerge for this purpose, so it is remarkable that we know how much stronger AD is at all, and it is even more remarkable that this degree of difference can be pinned down by large cardinals. This speaks to their utility in gauging mathematical power across conceptual domains, which may be viewed as a type of extrinsic justification for their consistency.

It would be astonishing if the theory of determinacy, developed largely independently of large cardinal considerations, all turned out to be vacuous. Indeed, it would nullify decades of some of the most involved mathematical investigation in all of human history. Why else would this extraordinarily intricate theory ``be there" if it were not consistent? 

The process of justifying any new proposed axiom, including large cardinal axioms, must rest on both intrinsic and extrinsic considerations, which can be thought of as corresponding to ``plausibility" and ``fruitfulness," respectively.\footnote{See Maddy's [18] for a thorough treatment of this idea.} Once we are comfortable with the idea that there is an infinite hierarchy of infinities, it does not seem that on the basis of their vast size alone could large cardinal axioms be said to violate intrinsic plausibility considerations. And we have already considered reasons why large cardinal axioms may be extrinsically justified. One more consideration in this favor is that forcing \textit{from} certain enormous large cardinals has proven to be a very useful technique in practice to set theorists.

On the other hand, one should be slightly cautious about using utility as a basis for belief in consistency. There is no doubt that large cardinal axioms are useful in proving things (as Scott said, ``if you want more, you have to assume more"). But if certain ones are inconsistent, then this should not be surprising, because then these axioms are useful without bounds by the principle of explosion. The naysayer can never be put entirely to rest.

A final thought: even if we believe large cardinals are consistent and thus (as a corollary of Gödel's Completeness theorem) believe that there exists \textit{some} model, it would be nice to be able to point to a specific canonical model. This is why the Inner Model Program, spearheaded by Woodin, is crucially important. This is discussed more in the final section, but the relevant point here is that it is potentially much more reasonable to believe in the consistency of those large cardinals for which the inner model problem has been solved than those for which it has not.\footnote{Of course, solving the inner model problem for some large cardinal is not a proof of its consistency because it may be argued that, as with V, our conception of such models and what makes them canonical is still inherently vague. Furthermore, formally all this establishes is that the consistency of that type of cardinal in V implies the consistency of that type of cardinal in the inner model. But in practice, the general pattern to inner model constructions performed by set theorists is fairly well-established and ameliorates some of these canonicity concerns.}

\section{The prospects for the Continuum Hypothesis (Or, why taking the Ultimate-L would be a W)}

Recall that CH says $2^{\aleph_0} = \aleph_1$ and GCH says $\delta^+ = 2^\delta$, while large cardinal axioms postulate the existence of enormous infinite sets.\footnote{Thankfully ``enormous," unlike ``huge" or ``large," is not a reserved word in this context.} Despite all the fuss we have made over large cardinal axioms and their mathematical strength, no large cardinal axiom has been shown to resolve CH in one way or another. In fact, by a theorem of Levy and Solovay, it is provably the case that (in a certain precise sense) no large cardinal axiom could resolve CH.\footnote{This is because CH is a $\Sigma_1^2$ statement, which by the Levy-Solovay theorem are not $\Omega$-complete, meaning that they can always have their truth-value altered by forcing.} To resolve it, we must look for an entirely new kind of axiom.

Before we can consider possible resolutions to CH, we briefly acknowledge two competing conceptions of set theory that we alluded to earlier. On the ``universe" view, there is a unique absolute background conception of set, and every statement like CH has a definitive truth value in it. On the ``multiverse" view, advocated by Hamkins, there are many distinct conceptions of set, each instantiated in one of many different set-theoretic universes which exhibit varying truth values for certain propositions. This is a form of higher-order Platonic realism about universes, but even on this view we may still prefer certain universes over others. Hamkins contends that due to our vast experience working in those models of set theory where CH is false, any ``dream solution" proposing a pre-reflective principle that would resolve CH in one direction or the other is untenable. Hamkins acknowledges that Woodin's current research program and corresponding solution template for resolving CH affirmatively (which we discuss below) is different from the others, but Hamkins argues that, as attractive as Woodin's conjectured CH paradise may be, it cannot just instantly render our experience in models of $\neg CH$ as confused and misguided.

It is entirely possible that Hamkin's position is the right one to hold, or that Woodin's program does not contradict it outright. But even Hamkins acknowledges that we may have reasons for preferring certain universes over others, and we can view the goal of Woodin's program as merely to paint a particularly elegant picture of the universe. Therefore, we shall proceed to consider a possible path forward that Woodin and proponents of CH would deem a resounding win for set theory.

To do so, we must discuss, at only a very high level, the general pattern of inner model construction. Recall Scott's Theorem: if V=L, then there does not exist a measurable cardinal. When set theorists construct inner models of large cardinals, they are seeking to find generalizations of L (using notions intermediary between the restrictive ``definable" powerset and the recklessly-behaved full powerset) that are compatible with ``large" large cardinals. This program has been remarkably successful at the level of measurable cardinals, Woodin cardinals, infinitely many Woodin cardinals, and beyond. But what occurs in practice (based on the current standard techniques) is that inner model construction occurs in a piecewise inductive manner: the inner model of a single cardinal of type A will not in general be an inner model of any stronger cardinal, nor even an inner model of two cardinals of type A. Thus, solving the inner model problem for, or deriving an inconsistency from, any large cardinal hypotheses would typically require solving the inner model problem at \textit{all} lower levels of the consistency strength hierarchy first. Or so was the case, until quite recently...

The extent to which the most spectacular results in modern set theory can be ascribed to one individual, Hugh Woodin, is truly mind-boggling. What Woodin recently showed (and which we shall not even hope to make precise) is the following: if the inner model program can be extended to the level of one supercompact cardinal, then it will essentially immediately subsume every stronger large cardinal hypothesis for ``free." Stated more cleanly, unlike how measurable cardinals cannot exist in L, it would be consistent for all known large cardinals to exist in the inner model of a supercompact cardinal (assuming they are consistent in the first place). Thus, a central focus of set theorists has become to solve the inner model program for one supercompact cardinal. If this can be done, the first bewildering consequence is that it will establish Reinhardt cardinals are inconsistent with ZF, resolving a problem that has been open since Kunen's gave his original inconsistency proof. The second remarkable consequence concerns the properties of the (conjectured) model itself, which has been termed ``Ultimate-L." As noted, Ultimate-L would be \textit{immune} to analogues of Scott's theorem and thus compatible with all known large cardinal axioms. But furthermore, Ultimate-L would also be \textit{immune} to independence by forcing, meaning in particular that it would be a model of CH.\footnote{Furthermore, GCH would hold up at least through some high level of the large cardinal hierarchy, though it is unknown whether Ultimate-L would be a model of GCH in its entirety.}

The hypothesis of Ultimate-L's existence can be made into a specific combinatorial statement. Actually, what we give below would be a corollary of the Ultimate-L conjecture, but one strong enough to still prove the inconsistency of Reinhardt cardinals with ZF.\footnote{As a technical quibble, we should note that the Ultimate-L conjecture is not merely the statement that the model Ultimate-L exists, though it may as well be taken as such for the purposes of our exposition.}

\begin{defn}
\textbf{HOD} is the proper class of all \textit{hereditarily ordinal definable} sets.
\end{defn}

That is, HOD consists of those sets that can be defined using ordinal parameters ``all the way down." A superset of L that also satisfies AC, HOD is a rich inner model. However, HOD's properties are far less understood than L's, and it is more plausible for HOD to be close to V.

\begin{thm}
Assume the Ultimate-L conjecture. Then the collection of regular cardinals which are not measurable cardinals in HOD constitute a proper class.
\end{thm}

If the Ultimate-L conjecture can be proven true, we would have a natural candidate axiom, V=Ultimate-L, which would be extraordinarily more appealing than the naive Axiom of Constructibility because it would preserve the rich landscape of large cardinals.

But the problems don't entirely go away. In a future where the Ultimate-L conjecture is ultimately proven true, there will actually be multiple natural candidate models with the two desired immunity properties. Which means that unless something else remarkable happens, adjudication will still be in order—and the philosophy of set theory will remain a flourishing discipline.

\section*{Acknowledgments}  First and foremost, I would like to thank Professor Jonathan Weinstein for his excellent mentorship this summer and over the prior semester. He has truly gone above and beyond, dedicating an extraordinary amount of time and effort to helping me learn about set theory. Any clarity this paper provides about large cardinals ought to be attributed to him, and any deficiencies attributed to my inability to properly distill his insights. I also wish to thank Professor Renato Feres for organizing the Freiwald scholars seminar over the same time period and for giving me the opportunity to participate. Professor Karl Schaefer suggested the theme of the paper, and Professors Nic Koziolek and Brendan Juba nurtured my love for the philosophical aspects of the Continuum Hypothesis and mathematical logic in the first place. Finally, Lucas Strammello helped me through a semester each of set theory and theory of computation, explaining concepts whenever I found myself lost (which was quite frequently). I am grateful for the support I have received from so many different places.

\end{document}